\documentclass[12pt]{article}
\usepackage{amssymb,amsthm,amsmath}
\usepackage[british]{babel}
\usepackage[all]{xy}
\usepackage{txfonts}
\usepackage{graphicx}
\usepackage{hyperref}
\usepackage{mathbbol}
\usepackage{mathabx}
\usepackage{calrsfs}
\usepackage[a4paper,text={128mm,185mm}]{geometry}

\theoremstyle{plain}
\newtheorem{theorem}[subsection]{Theorem}
\newtheorem{lemma}[subsection]{Lemma}
\newtheorem{proposition}[subsection]{Proposition}
\newtheorem{corollary}[subsection]{Corollary}

\theoremstyle{definition}
\newtheorem{definition}[subsection]{Definition}
\newtheorem{remark}[subsection]{Remark}

\newtheorem{examples}[subsection]{Examples}

\newcommand{\comp}{\raisebox{0.2mm}{\ensuremath{\scriptstyle{\circ}}}}

\newcommand{\defn}{\textbf}
\newcommand{\noproof}{\hfill \qed}
\newcommand{\To}{\Rightarrow}

\newcommand{\Aut}{\ensuremath{\mathrm{Aut}}}
\DeclareMathOperator{\cod}{Cod}

\newcommand{\dis}{\ensuremath{\mathrm{Dis}}}
\newcommand{\Gal}{\mathrm{Gal}_\Gamma}
\renewcommand{\H}{\mathrm{H}}
\DeclareMathOperator{\Hom}{Hom}
\DeclareMathOperator{\Ker}{Ker}
\DeclareMathOperator{\op}{op}
\newcommand{\R}[1]{\mathrm{Eq}(#1)}

\newcommand{\C}{\ensuremath{\mathcal{C}}}
\newcommand{\Dc}{\ensuremath{\mathcal{D}}}
\newcommand{\E}{\ensuremath{\mathcal{E}}}
\newcommand{\F}{\ensuremath{\mathcal{F}}}
\newcommand{\M}{\ensuremath{\mathcal{M}}}
\newcommand{\N}{\ensuremath{\mathcal{N}}}
\newcommand{\X}{\ensuremath{\mathcal{X}}}

\newcommand{\Ab}{\ensuremath{\mathsf{Ab}}}
\newcommand{\Arr}{\ensuremath{\mathsf{Arr}}}

\newcommand{\Ext}{\ensuremath{\mathsf{Ext}}}

\newcommand{\Gp}{\ensuremath{\mathsf{Gp}}}
\newcommand{\NExt}{\ensuremath{\mathsf{NExt}_\Gamma}}

\newcommand{\Set}{\ensuremath{\mathsf{Set}}}

\newcommand{\LoCo}{\ensuremath{\mathsf{LoCo}}}

\def\pullback{% with thanks to Valerian Even
 \ar@{-}[]+R+<6pt,-1pt>;[]+RD+<6pt,-6pt>%
 \ar@{-}[]+D+<1pt,-6pt>;[]+RD+<6pt,-6pt>}
 
\def\pushout{%
 \ar@{-}[]+L+<-6pt,1pt>;[]+LU+<-6pt,6pt>%
 \ar@{-}[]+U+<-1pt,6pt>;[]+LU+<-6pt,6pt>}

\def\splitpullback{%
 \ar@{-}[]+R+<6pt,-.51ex>;[]+RD+<6pt,-6pt>%
 \ar@{-}[]+D+<.51ex,-6pt>;[]+RD+<6pt,-6pt>}

\def\skewpullback{%
 \ar@{-}[]+LD+<-6pt,-6pt>;[]+LDD+<-6pt,-15.5pt>%
 \ar@{-}[]+D+<-1pt,-6pt>;[]+LDD+<-6pt,-15.5pt>}
 
\newdir{>>}{{}*!/3.5pt/:(1,-.2)@^{>}*!/3.5pt/:(1,+.2)@_{>}*!/7pt/:(1,-.2)@^{>}*!/7pt/:(1,+.2)@_{>}}
\newdir{ >>}{{}*!/8pt/@{|}*!/3.5pt/:(1,-.2)@^{>}*!/3.5pt/:(1,+.2)@_{>}}
\newdir{ |>}{{}*!/-3.5pt/@{|}*!/-8pt/:(1,-.2)@^{>}*!/-8pt/:(1,+.2)@_{>}}
\newdir{ >}{{}*!/-8pt/@{>}}
\newdir{>}{{}*:(1,-.2)@^{>}*:(1,+.2)@_{>}}
\newdir{<}{{}*:(1,+.2)@^{<}*:(1,-.2)@_{<}}

\title{\large\bf THE FUNDAMENTAL GROUP FUNCTOR\\ AS A KAN EXTENSION\footnotetext{The first author's research was supported by Fonds voor Wetenschappelijk Onderzoek (FWO-Vlaanderen). The third author works as \emph{charg\'e de recherches} for Fonds de la Recherche Scientifique--FNRS. Both would like to thank the DPMMS for its kind hospitality during their stays in Cambridge.}}

\author{\em dedicated to Ren\'e Guitart\\ \em on the occasion of his sixty-fifth birthday}

\begin{document}

\date{}

\maketitle

\vspace*{-8mm}

\begin{center} {\textbf{\em by Tomas EVERAERT, Julia GOEDECKE and Tim VAN DER LINDEN}}
\end{center}

\vspace*{2mm}

\begin{minipage}{118mm}{\small
\textbf{R\'esum\'e.} On montre que le foncteur \emph{groupe fondamental} consi\-d\'er\'e en th\'eorie de Ga\-lois cat\'egorique peut \^etre calcul\'e comme une extension de Kan.

\smallskip
\textbf{Abstract.} We prove that the fundamental group functor from categorical Galois theory may be computed as a Kan extension.

\medskip
\textbf{Keywords.} Homology, categorical Galois theory, semi-abelian category, higher central extension, satellite, Kan extension, fundamental group

\smallskip
\textbf{Mathematics Subject Classification (2010).} 18G50, 18G60, 18G15, 20J, 55N
}\end{minipage}

\thispagestyle{empty}

%\tableofcontents

\section{Introduction}
The main aim of this paper is to prove that the fundamental group from categorical Galois theory \cite{Janelidze:Hopf} may be computed as a Kan extension:
\begin{equation}\label{finalKanextension}
\vcenter{\xymatrix{& \NExt(\C) \ar@{}[d]|-{\Nearrow}_(.45){\delta} \ar[ld]_-{\cod} \ar[rd]^-{\Ker} \\
\C \ar@{.>}[rr]_-{\pi_{1}(-,I)} && \X }}
\end{equation}
This makes it a \emph{satellite} in the sense of Janelidze~\cite{Janelidze-Satellites}, Guitart--Van den Bril~\cite{Guitart-Bril, Guitart:Anabelian} and two authors of the present paper~\cite{GVdL2}. Here \(\Gamma\) is a Galois structure, consisting of an adjunction \(I\dashv H\colon \C\to \X\) and certain classes of morphisms, \(\NExt(\C)\) is the category of \emph{normal extensions}, which are defined via the Galois structure \(\Gamma\), $\Ker$ is the kernel functor and $\cod$ is the codomain functor. 

In fact, we will see this in two steps. First we show that the following is a Kan extension: 
\begin{equation}\label{intermediateKanextension}
\vcenter{\xymatrix{& \NExt(\C) \ar@{}[d]|-{\Nearrow}_{\kappa} \ar[ld]_-{\cod} \ar[rd]^-{\Gal(-,0)} \\
\C \ar@{.>}[rr]_-{\pi_{1}(-,I)} && \Gp(\X) }}
\end{equation}
Here \(\Gal(-,0)\) gives the Galois group of a normal extension, as defined in the context of categorical Galois theory by Janelidze \cite{Janelidze:Hopf}. This step uses that the Galois group functor is a \emph{Baer invariant} with respect to the codomain functor, in the following sense: any two morphisms between objects in \(\NExt(\C)\) which agree on the codomain of the objects are sent to the \emph{same} morphism between the Galois groups. This makes it possible to define \(\pi_1(B,I)\) by taking a \emph{weakly universal} normal extension \(u\colon {U\to B}\) of \(B\), and then applying the Galois group functor to it. The above property ensures that this assignment is well defined, i.e.~independent of the choice of \(u\), and functorial in \(B\). 

To attain the first-mentioned Kan extension from this one, we use the fact that the underlying object of the Galois group of a normal extension \(p\colon {E\to B}\) can be computed as the intersection of the kernel of \(p\) with the kernel of the unit \(\eta_E\colon {E\to HI(E)}\). This makes it a subobject of \(\Ker(p)\), and so gives a component-wise monic natural transformation \(\iota\colon{\Gal(-,0)\To \Ker}\). We then show that, for any given functor \(F\colon {\C\to \X}\), any natural transformation \({F\comp \cod\To \Ker}\) lifts over this \(\iota\). This implies that the universal property of the Kan extension~\eqref{intermediateKanextension} carries over to \eqref{finalKanextension}.

Our arguments go through under fairly weak assumptions on the Galois structure \(\Gamma\), and can moreover be adapted to situations where the fundamental group functor is not everywhere defined. In the latter case, we obtain a Kan extension similar to \eqref{finalKanextension} and \eqref{intermediateKanextension}, by replacing \(\C\) with its full subcategory of objects \(B\) for which \(\pi_1(B,I)\) \emph{is} defined, and restrict \(\NExt(\C)\) accordingly.

When \(\C\) is pointed, exact and Mal'tsev, and \(\X\) is a Birkhoff subcategory of~\(\C\), we show that \eqref{finalKanextension} induces a Kan extension
 \[
\xymatrix{& \Ext_{\Gamma}(\C) \ar@{}[d]|-{\Nearrow}_(.45){} \ar[ld]_-{\cod} \ar[rd]^-{\Ker\circ I_1} \\
\C \ar@{.>}[rr]_-{\pi_{1}(-,I)} && \X }
\]
where \(\Ext_{\Gamma}(\C)\) is the category of regular epimorphisms (=extensions) in \(\C\), and~\(I_1\) is left adjoint to the inclusion functor \(\NExt(\C)\to \Ext_{\Gamma}(\C)\). In the case of a semi-abelian \(\C\), this Kan extension was first obtained in \cite{GVdL2}, where it was also shown that, for a given extension \(p\), the \(p\)-component of the universal natural transformation defining it is a connecting homomorphism in the long exact homology sequence induced by \(p\). 

The latter result, we will see, has a topological counterpart: for a certain Galois structure, the components of the universal natural transformation \(\delta\) defining the Kan extension \eqref{finalKanextension} (or, actually, the ``restricted'' version, since here the fundamental group functor is not everywhere defined) are connecting maps in an exact \emph{homotopy} sequence. 

Note that we have used the same notation \(\pi_1(-,I)\) for functors \(\C\to \Gp(\X)\) and \(\C\to \X\) and have called both ``fundamental group functor'', while the image of an object \(B\in|\C|\) under the latter is actually the \emph{underlying object} of the fundamental group \(\pi_1(B,I)\). A similar remark can be made regarding the Galois group functor $\Gal(-,0)$. This does not pose any problems when \(\X\) is Mal'tsev, since then any internal group is determined, up to isomorphism, by its underlying object. However, the latter is of course not true in general, and it is in particular false for the topological example just referred to. 

\section{Galois structures}
To define the ingredients of the Kan extensions considered in this paper, we need a \emph{Galois structure} and the concept of \emph{normal extension} arising from it, as introduced by Janelidze~\cite{Janelidze:Pure, Janelidze:Recent}.

\begin{definition} A \defn{Galois structure} \(\Gamma=(\C,\X,H,I,\eta, \epsilon, \E, \F)\) on a category \(\C\) consists of an adjunction
\[
\xymatrix@1{\C \ar@<1ex>[r]^I \ar@{}[r]|\bot & \X \ar@<1ex>[l]^H}
\]
with unit \(\eta\colon {1_{\C}\To HI}\) and counit \(\epsilon\colon {IH\To 1_{\X}}\), as well as classes of morphisms \(\E\) in \(\C\) and \(\F\) in \(\X\) such that
\begin{enumerate}
 \item \(\E\) and \(\F\) contain all isomorphisms;
 \item \(\E\) and \(\F\) are pullback-stable, meaning here that the pullback of a morphism in \(\E\) (resp.\ \(\F\)) along any morphism \emph{exists} and is in \(\E\) (resp.\ \(\F\));
 \item \(\E\) and \(\F\) are closed under composition;
 \item \(H(\F)\subseteq \E\);
 \item \(I(\E)\subseteq \F\).
\end{enumerate}
We will use the terminology of \cite{Janelidze:Recent} and call the morphisms in  \(\E\) \defn{fibrations}.
\end{definition}

Given such a Galois structure, some fibrations have some additional useful and interesting properties. We write \((\E\downarrow B)\) for the full subcategory of the slice category \((\C\downarrow B)\) determined by morphisms in \(\E\).

\begin{definition}
 A \defn{trivial covering} is a morphism \(f\colon{A\to B}\) in \(\E\) such that
 \[
 \xymatrix{A \ar[r]^-{\eta_A} \ar[d]_f & HI(A) \ar[d]^{HI(f)}\\
 B \ar[r]_-{\eta_B} & HI(B)}
 \]
is a pullback. A \defn{monadic extension} is a fibration \(p\colon {E\to B}\) such that the pullback functor \(p^*\colon {(\E\downarrow B)\to (\E\downarrow E)}\) is monadic. A \defn{covering} (sometimes called \defn{central extension}) is a fibration \(f\colon {A\to B}\) whose pullback \(p^*(f)\) along \emph{some} monadic extension \(p\) is trivial. A \defn{normal extension} is a monadic extension \(p\) such that \(p^*(p)\) is a trivial covering, i.e.\ a monadic extension with trivial kernel pair projections.
\end{definition}

The trivial coverings are exactly those fibrations which are cartesian with respect to the functor \(I\colon {\C\to \X}\). 

For many uses of such Galois structures, we need \(\Gamma\) to satisfy an extra property called \emph{admissibility}. For this we consider the induced adjunction
\[
 \xymatrix@1{(\E\downarrow B) \ar@<1ex>[r]^-{I^B} \ar@{}[r]|-\bot & (\F\downarrow I(B)) \ar@<1ex>[l]^-{H^B}}
\]
for any object \(B\in\C\); here \(I^B\colon {(\E\downarrow B)\to (\F\downarrow I(B))}\) is the restriction of \(I\), and \(H^B\) sends a fibration \(g\colon {X\to I(B)}\) to the pullback of \(H(g)\) along \(\eta_B\):
\[
\xymatrix{A \ar[r]^-{\eta_A} \ar[d]_-{H^B(g)} \pullback & H(X) \ar[d]^-{H(g)}\\
B \ar[r]_-{\eta_B} & HI(B)}
\]

\begin{definition}
 A Galois structure \(\Gamma=(\C,\X,H,I,\eta,\epsilon,\E,\F)\) is \defn{admissible} when all functors \(H^B\) are full and faithful. 
\end{definition}

An important consequence of admissibility is
\begin{lemma}\cite[Proposition 2.4]{JK:Reflectiveness}\label{Lemma I preserves pullbacks along trivial}
 If \(\Gamma\) is admissible, then \(I\colon {\C\to \X}\) preserves pullbacks along trivial coverings. In particular, the trivial coverings are pullback-stable.\noproof
\end{lemma}

So if the Galois structure is admissible, we can view the class of all trivial coverings as the pullback-closure of \(H(\F)\), while the coverings are \emph{locally trivial}. In certain situations the coverings are also pullback-stable:

\begin{lemma}\label{Lemma Pullback Stable Normal}
If \(\Gamma\) is admissible and monadic extensions are pullback-stable, then normal extensions and coverings are pullback-stable.
\end{lemma}
\begin{proof}
The proof of \cite[Proposition 4.3]{Janelidze-Kelly} remains valid under our assumptions.
\end{proof}

\begin{examples}\label{Examples}
There are many different kinds of categorical Galois structures; we list a few which are relevant for the present article.
\begin{enumerate}
\item\label{groups} Take \(\C=\Gp\) and \(\X=\Ab\), the subcategory of abelian groups in the category of groups, and let \(I\) be the abelianisation functor sending a group~\(G\) to the quotient \(G/[G,G]\), which is left adjoint to the inclusion \(H\). Then choosing \(\E\) and \(\F\) to be the classes of surjective group homomorphisms defines an admissible Galois structure~\(\Gamma\) as above. Here every map in \(\E\) is a monadic extension, the trivial coverings are those surjective homomorphisms \(A\to B\) whose restriction to the commutator subgroups \([A,A]\to [B,B]\) is an isomorphism, and the coverings are the central extensions in the usual sense: surjective homomorphisms whose kernel lies in the centre of the domain. Normal extensions and coverings coincide. (See \cite{Janelidze:Pure}.)

\item\label{Mal'tsev} More generally, taking for \(\C\) an exact Mal'tsev (or Goursat) category and for \(\X\) a Birkhoff subcategory (= a full reflective subcategory closed under subobjects and regular quotients), and all regular epimorphisms for \(\E\) and~\(\F\), defines an admissible Galois structure \(\Gamma\), whose coverings are studied in \cite{Janelidze-Kelly}. Normal extensions and coverings still coincide, and every regular epimorphism is a monadic extension. In particular, \(\C\) could be a Mal'tsev \emph{variety} and \(\X\) its subvariety of abelian algebras, in which case the coverings are the central extensions arising from commutator theory in universal algebra: those surjective homomorphisms \(f\colon A\to B\) for which the commutator \([\R{f}, A\times A]\) of the kernel congruence \(\R{f}\) of \(f\) with the largest congruence \(A\times A\) on \(A\) is trivial (see \cite{Janelidze-Kelly:Maltsev,Gran-Alg-Cent}). Or, \(\C\) could be a variety of \(\Omega\)-groups~\cite{Higgins} and \(\X\) an 
arbitrary subvariety of \(\C\). Now the coverings are the (relative) central extensions studied by Fr\"ohlich and others (see \cite{Janelidze-Kelly}). 

\item\label{topological} Consider \(\C=\LoCo\) to be the category of locally connected topological spaces and \(\X=\Set\) the category of sets. Take \(I=\pi_0\), the connected components functor, \(H=\dis\) the discrete topology functor, \(\E\) the class of \'etale maps (= local homeomorphisms), and \(\F\) the class of all maps in \(\Set\). This gives another admissible Galois structure. Here the monadic extensions are exactly the \emph{surjective} local homeomorphisms, the trivial coverings and the coverings are, respectively, the disjoint unions of trivial covering maps, and the covering maps, in the usual topological sense. For connected \(A\) and \(B\), a normal extension \(f\colon A\to B\) is the same as a regular covering map: a covering map \(f\colon A\to B\) such that for every pair of elements \(x\), \(y\in A\) which are in the same fibre of \(f\) there is a unique continuous map \(a\colon A\to A\) (actually, a covering) such that \(f=f\comp a\) and \(a(x)=y\). See~\cite[Chapter 6]{Borceux-Janelidze} for more details.

\item\label{simplicial} Similarly, take \(\C\) to be the category of simplicial sets and \(\X=\Set\) with the adjunction consisting of \(I=\pi_0\) and \(H\) giving the discrete simplicial set on a given set. Then taking \(\E\) and \(\F\) to be the classes of all morphisms gives an admissible Galois structure. For this example, monadic extensions are degree-wise surjective functions. The coverings are precisely the coverings in the sense of Gabriel--Zisman \cite{Gabriel-Zisman}: Kan fibrations whose ``Kan liftings'' are uniquely determined. See \cite[A.3.9]{Borceux-Janelidze} for more details.

\item\label{simplicial2} For a different Galois structure $\Gamma$ on the category $\C$ of simplicial sets, let $\X$ be the category of groupoids, and $I$ and $H$ be the fundamental groupoid and nerve functors, and take for $\E$ and $\F$ the classes of Kan fibrations, and of fibrations in the sense of Brown \cite{Brown:Fibrations}, respectively. This particular $\Gamma$ is studied in \cite{Brown-Janelidze:Second} where its covering morphisms are called \emph{second order covering maps}. It is \emph{not} admissible. 

\item\label{pointedtopological} Example \eqref{topological} has an obvious ``pointed'' version, obtained by replacing \(\LoCo\) and \(\Set\) by the categories \(\LoCo_*\) and \(\Set_*\) of pointed locally connected spaces and of pointed sets, respectively. \(\E\) and \(\F\) now consist of those \'etale maps and maps that preserve the basepoint. Clearly, this is still an (admissible) Galois structure; the monadic extensions, trivial coverings, coverings and normal extensions are ``the same'' as in the non-pointed case, only now they are required to be basepoint-preserving.

\item\label{fields} \emph{Categorical Galois theory} does indeed capture classical Galois theory, as the name suggests. For this, let \(k\) be some fixed field and take \(\C\) to be the dual of the category of finite-dimensional commutative \(k\)-algebras with \(\E^{\op}\) all algebra morphisms, \(\X\) the category of finite sets with \(\F\) the class of all functions, and \(I\colon\C\to \X\) defined through idempotent decomposition. See \cite[A.2]{Borceux-Janelidze} or \cite{Janelidze:Pure} for further details. 
\end{enumerate}
\end{examples}

For the rest of this paper, we will assume that our Galois structures are admissible and that \(H\) is in fact an inclusion of a full reflective subcategory \(\X\) into \(\C\). We will also assume that monadic extensions are pullback-stable. Note that this is the case for each of the examples above, with the exception of \eqref{simplicial2}.

One of the important concepts in categorical Galois theory is the \emph{Galois groupoid}:

\begin{definition}\cite{Janelidze:Pure, Janelidze:Hopf}
Let \(p\colon {E\to B}\) be a normal extension of \(B\). Then the \defn{Galois groupoid} \(\Gal(p)\) of \(p\) is the image under \(I\) of the kernel pair \(\R{p}\) of~\(p\).
\[
\xymatrix{ \R{p} \ar@<.5ex>[r]^-{d} \ar@<-.5ex>[r]_-{c} \ar[d]_-{\eta_{\R{p}}} & E \ar[d]^{\eta_{E}} \ar[r]^p & B\\
 I(\R{p}) \ar@<.5ex>[r]^-{I(d)} \ar@<-.5ex>[r]_-{I(c)} & I(E)}
\]
\end{definition}

Note that this image of the kernel pair is indeed a groupoid: since the functor \(I\) preserves pullbacks along trivial coverings (by Lemma \ref{Lemma I preserves pullbacks along trivial}), the image of any groupoid with trivial domain and codomain morphisms is again a groupoid (see the definition of groupoids~\ref{Definition-Groupoids}). And since \(p\) is normal, its kernel pair projections are indeed trivial coverings.

\section{Internal groupoids}\label{Section Groupoids}
We have already seen groupoids enter the picture above, so we recall the definition.

\begin{definition}\label{Definition-Groupoids}
An \defn{internal category} in a category \(\C\) is a diagram
\[
\xymatrix@C=3em{ R_1 \ar@<1ex>[r]^{d} \ar@<-1ex>[r]_{c} & R_0 \ar[l]|e}
\]
such that \(d e = 1_{R_0}=c e\), together with a \defn{multiplication} (or composition)
 \[m\colon {R_1\times_{R_0} R_1\to R_1}\] 
making the following diagrams commute, where the pullback \textbf{(1)} defines the object \(R_1\times_{R_0} R_1\) of ``composable arrows'':
\begin{equation*}\label{eq-Groupoid-Pullbacks}
 \vcenter{\xymatrix{R_1\times_{R_0} R_1 \ar@{}[dr]|{\textbf{(1)}} \pullback \ar[r]^-{p_1} \ar[d]_{p_2} & R_1 \ar[d]^{c}\\
 R_1 \ar[r]_{d} & R_0}}
 \qquad
\vcenter{ \xymatrix{R_1\times_{R_0} R_1 \ar[r]^-{m} \ar[d]_{p_1} \ar@{}[dr]|{\textbf{(2)}} & R_1 \ar[d]^{d}\\
 R_1 \ar[r]_{d} & R_0}}
 \qquad
 \vcenter{\xymatrix{R_1\times_{R_0} R_1 \ar[r]^-{m} \ar[d]_{p_2} \ar@{}[dr]|{\textbf{(3)}} & R_1 \ar[d]^{c}\\
 R_1 \ar[r]_{c} & R_0;}}
\end{equation*}
furthermore, the composition \(m\) makes the diagrams
\[
\resizebox{\textwidth}{!}{
\mbox{$\vcenter{\xymatrix{R_{1} \ar[r]^-{\langle{1_{R_1}, sc}\rangle} \ar@{=}[rd] & R_1\times_{R_{0}}R_{1} \ar[d]^-{m} & R_{1} \ar@{=}[ld] \ar[l]_-{\langle{sd,1_{R_1}}\rangle}\\
& R_{1}}}
\qquad \text{and}\qquad
\vcenter{\xymatrix{R_1\times_{R_0}R_1\times_{R_0}R_1 \ar[r]^-{1\times m} \ar[d]_{m\times 1} & R_1\times_{R_0}R_1 \ar[d]^m\\
R_1\times_{R_0}R_1 \ar[r]_-{m} & R_1}}$}}
\]
commute. An internal category \(R\) is an \defn{internal groupoid} when there exists a morphism \(s\colon R_1\to R_1\) such that \(ds=c\) and \(cs=d\) and both squares
\[
\xymatrix{
R_1 \ar[r]^-{\langle1_{R_1},s\rangle} \ar[d]_{d} & R_1\times_{R_0}R_1 \ar[d]^m \\
R_0 \ar[r]_e & R_1}
 \qquad \qquad \qquad
\xymatrix{
R_1 \ar[r]^-{\langle s,1_{R_1}\rangle} \ar[d]_{c} & R_1\times_{R_0}R_1 \ar[d]^m \\
R_0 \ar[r]_e & R_1}
\]
commute. Such an \(s\) is necessarily unique. In fact, it is well known that an internal category \(R\) is an internal groupoid if and only if \textbf{(2)} and \textbf{(3)} are also pullbacks.

An \defn{internal functor} between two internal categories \(R\) and \(S\) is a pair of morphisms \((f_0,f_1)\) making the three squares with \(d\), \(c\) and \(e\) as on the left
\begin{equation*}\label{internalfunctor}
 \vcenter{\xymatrix@=3em{R_1 \ar[r]^{f_1} \ar@<1ex>[d]^-c \ar@<-1ex>[d]_-d & S_1 \ar@<1ex>[d]^-c \ar@<-1ex>[d]_-d\\
 R_0 \ar[r]^{f_0} \ar[u]|-{e} & S_0 \ar[u]|-{e}}}
 \qquad\qquad\qquad
 \vcenter{\xymatrix{R_1\times_{R_0}R_1 \ar[r]^{f_1\times f_1} \ar[d]_m & S_1\times_{R_0} S_1 \ar[d]^m\\
 R_1 \ar[r]_{f_1} & S_1
 }}
\end{equation*}
as well as the right hand square commute.

An internal groupoid \(R\) with \(R_0=1\), the terminal object, is called an \defn{internal group}. We shall write \(\Gp(\C)\) for the category of internal groups and internal functors. 
\end{definition}

\begin{definition}[Internal natural transformations and isomorphisms]\label{Definition-Homotopy}
Given two internal functors \(f\), \(g\colon {R\to S}\) between internal categories \(R\) and \(S\), an \defn{internal natural transformation} from \(f\) to \(g\) is a morphism \(\mu \colon {R_0\to S_1}\) as in
\begin{equation*} \vcenter{\xymatrix@=3em{R_1 \ar@<.5ex>[r]^{f_1} \ar@<-.5ex>[r]_{g_1} \ar@<1ex>[d]^-c \ar@<-1ex>[d]_-d & S_1 \ar@<1ex>[d]^-c \ar@<-1ex>[d]_-d\\
 R_0 \ar@{.>}[ru]|{\mu} \ar@<.5ex>[r]^{f_0} \ar[u]|-{e} \ar@<-.5ex>[r]_{g_0} & S_0 \ar[u]|-{e}}}
\end{equation*}
satisfying
\[
\begin{minipage}{.4\textwidth}
\begin{enumerate}
 \item \(d\mu=f_0\),
 \item \(c\mu=g_0\),
 \item\label{natural-trans} \(m\langle{f_1,\mu c}\rangle=m\langle{\mu d, g_1}\rangle\).
\end{enumerate}
\end{minipage}
\qquad\qquad
\vcenter{\xymatrix{R_1 \ar[r]^-{\langle{f_1,\mu c}\rangle} \ar[d]_-{\langle{\mu d, g_1}\rangle} & S_1\times_{S_0} S_1 \ar[d]^-m\\
 S_1\times_{S_0} S_1 \ar[r]_-m & S_1}}
\]
For fixed internal categories \(R\) and \(S\), the internal functors \(R\to S\) and the internal natural transformations between them form a category: the composition of two natural transformations \(\mu\colon f\to g\) and \(\nu\colon g\to h\) is given by the morphism \(m\langle \nu,\mu\rangle\); the identity on \(f\) is given by the morphism \(ef_0\). In particular, an internal natural transformation \(\mu\) is an \defn{internal natural isomorphism} when there is a (unique) internal natural transformation \(\nu\) from \(g\) to~\(f\) such that \(m\langle{\mu,\nu}\rangle=ef_0\) and \({m\langle{\nu,\mu}\rangle=eg_0}\).
\end{definition}

\begin{remark}
When \(R\) and \(S\) are internal groupoids, an internal natural transformation is automatically a natural isomorphism between \(f\) and \(g\).
\end{remark}

\begin{remark}
 If \(S\) is a relation, then \(d\) and \(c\) are jointly monic, so (iii) is automatically satisfied.
\end{remark}

In particular, for effective equivalence relations we have

\begin{lemma}\label{Lemma-Induced-Homotopy}
 Given two morphisms \(f=(f_1,f_0)\) and \(g=(g_1,g_0)\) from \(b\colon {B_1\to B_0}\) to \(c\colon {C_1\to C_0}\) satisfying \(f_0=g_0\), there is an internal natural isomorphism between the induced internal functors from~\(\R{b}\) to \(\R{c}\).
\end{lemma}
\begin{proof}
The condition \(f_0=g_0\) implies that \(cf_1=f_0b=g_0b=cg_1\). So let \(\mu=\langle{f_1,g_1}\rangle\). Then (i) and (ii) from Definition~\ref{Definition-Homotopy} are satisfied by definition, and (iii) is satisfied automatically, as \(\R{c}\) is a kernel pair and so a relation.
\end{proof}

In the special case that \(B_0=C_0=A\) and \(f_0=1_A\), we say that \(f\) is a \defn{morphism over~\(A\)}.

From now on, let \(\C\) be a finitely complete pointed category. Then for any groupoid \(R\) in \(\C\), we may restrict $R_0$ to the zero object $0$,
and $R_1$ to  \(\Ker(d)\cap\Ker(c)\), which gives us the internal group of ``loops at \(0\)'' or ``internal automorphisms at~\(0\)'', which we denote by \(\Aut_R(0)\). When we restrict to this group of internal automorphisms, natural isomorphisms as above collapse the two functors onto each other:

\begin{lemma}\label{Lemma-Same-On-Intersection}
Any two naturally isomorphic functors \(f\), \(g\colon {R\to S}\) between internal categories induce the same morphism~\({\Aut_R(0)\to \Aut_S(0)}\).
\end{lemma}
\begin{proof}
Consider the diagram
\[
\xymatrix@=3em{\Ker(d)\cap\Ker(c) \ar@<.5ex>[r]^-{\overline{f}} \ar@<-.5ex>[r]_-{\overline{g}} \ar[d]_-k & \Ker(d)\cap \Ker(c) \ar[d]^-l\\
R_1 \ar@<.5ex>[r]^-{f_1} \ar@<-.5ex>[r]_-{g_1} \ar@<1ex>[d]^-{c} \ar@<-1ex>[d]_-{d} & S_1 \ar@<1ex>[d]^-{c} \ar@<-1ex>[d]_-{d}\\
R_0 \ar@<.5ex>[r]^-{f_0} \ar[u]|-{e} \ar@<-.5ex>[r]_-{g_0} \ar@{.>}[ur]|-{\mu} & S_0 \ar[u]|-{e}}
\]
in which \(k\) and \(l\) are the inclusions of \(\Ker(d)\cap\Ker(c)\) into \(R_1\) and \(S_1\), respectively.
We wish to show that \(\overline{f}=\overline{g}\), or equivalently, that \(l\overline{f}=l\overline{g}\), as \(l\) is a monomorphism. From Condition~(iii) we know that \(m\langle{f_1k,\mu c k}\rangle=m\langle{\mu d k, g_1 k}\rangle\). But since \(dk=0=ck\) and $dl=0=cl$, we can reformulate this as
\begin{align*}
 m\langle{f_1k,\mu c k}\rangle&=m\langle{l\overline{f},ecl\overline{f}}\rangle=m\langle{1_{S_1},ec}\rangle l\overline{f}=l\overline{f},\\
 m\langle{\mu d k,g_1k}\rangle&=m\langle{edl\overline{g},l\overline{g}}\rangle=m\langle{ed,1_{S_1}}\rangle l\overline{g}=l\overline{g}
\end{align*}
giving \(l\overline{f}=l\overline{g}\) as required.
\end{proof}

\section{The Galois group and the fundamental group}

Let \(\Gamma=(\C,\X,H,I,\eta, \epsilon, \E, \F)\) be an admissible Galois structure on a finitely complete pointed category \(\C\) with \(H\) a full inclusion, and assume that monadic extensions are pullback stable. Note that this excludes the classical Galois theory Example \ref{Examples} \eqref{fields}, but it includes Examples \ref{Examples} \eqref{groups} and \eqref{pointedtopological}, as well as all the Galois structures of Example \ref{Examples} \eqref{Mal'tsev} for which \(\C\) is pointed.

\begin{definition}\label{Definition Gal}\cite{Janelidze:Hopf}
For a normal extension \(p\colon{E\to B}\), its \defn{Galois group}
\[
\Gal(p,0)=\Aut_{\Gal(p)}(0)
\]
is the group of automorphisms at \(0\) of the Galois groupoid: 
\[
\xymatrix{& \R{p} \splitpullback \ar@<.5ex>[r]^-{d} \ar@<-.5ex>[r]_-{c} \ar[d]_-{\eta_{\R{p}}} & E \ar[d]^{\eta_{E}} \ar[r]^p & B\\
\Gal(p,0)=\Ker(Id)\cap \Ker(Ic) \ar[r] & I(\R{p}) \ar@<.5ex>[r]^-{I(d)} \ar@<-.5ex>[r]_-{I(c)} & I(E)}
\]
\end{definition}

The resulting functor 
\[
 \Gal(-,0)\colon {\NExt(\C)\to \Gp(\X)}
\]
has some very useful properties: it is a \emph{Baer invariant}~\cite{EverVdL1, Froehlich} with respect to the codomain functor \(\cod\colon {\NExt(\C)\to \C}\), in the sense that any two maps between normal extensions which agree on the codomains also induce the same map between the Galois groups. To show this, we will use some properties of Section~\ref{Section Groupoids}.

\begin{lemma}\label{Lemma Preservation of Homotopies}
If two internal functors \(f\), \(g\colon {R\to S}\) between internal categories with source and target morphisms \(d\), \(c\) being trivial coverings are naturally isomorphic, then the functors \(I(f)\), \(I(g)\colon {I(R)\to I(S)}\) are still naturally isomorphic.
\end{lemma}
\begin{proof}
Recall that \(I\) preserves pullbacks along trivial coverings, so \(I(R)\) and \(I(S)\) are still internal categories. In particular, \[I(S_1\times_{S_0}S_1)={I(S_1)\times_{I(S_0)} I(S_1)}\] and \(I(m)\) is the multiplication of \(I(S)\). 

Let \(\mu \colon {R_0\to S_1}\) be an internal natural isomorphism between \(f\) and \(g\). Then functoriality of \(I\) and the preservation of the multiplication ensures that \(I(\mu)\) is still an internal natural transformation.
\end{proof}

\begin{proposition}\label{Proposition-Baer-Invariant}
Let \(p\colon {E\to B}\) and \(p'\colon {E'\to B'}\) be normal extensions. Any two morphisms \((f,b)\colon p\to p'\) and \((g,b)\colon{p\to p'}\) in \(\NExt(\C)\) with the same codomain component induce the same morphism
\[
{\Gal(p,0)\to \Gal(p',0)}
\]
on the Galois groups.
\end{proposition}
\begin{proof}
 This follows from Lemmas~\ref{Lemma-Induced-Homotopy}, \ref{Lemma Preservation of Homotopies} and \ref{Lemma-Same-On-Intersection}.
\end{proof}

In particular, this means that any endomorphism \((f,1_B)\colon {p\to p}\) induces the \emph{identity} on the Galois group \(\Gal(p,0)\). This means that we can now sensibly introduce the following definition. Recall that a normal extension \(u\colon U\to B\) is called \defn{weakly universal} if it is a weak initial object in the full subcategory \(\NExt(B)\) of \((\C\downarrow B)\) given by all normal extensions of \(B\), i.e.~for every normal extension \(p\colon E\to B\) there exists a morphism \(e\colon U\to E\) such that \(p\comp e=u\).

\begin{definition}\label{Definition Fundamental group}\cite{Janelidze:Hopf}
Given an object \(B\) of \(\C\), its \defn{fundamental group (with coefficient functor \(I\))} is the Galois group
\[
 \pi_{1}(B,I)=\Gal(u,0)
\]
of some weakly universal normal extension \(u\colon{U\to B}\), assuming such exists.

\end{definition}
Note that \(\pi_1(B,I)\) is independent of the choice of weakly universal normal extension \(u\colon {U\to B}\), by Proposition~\ref{Proposition-Baer-Invariant} and weak universality of \(u\). Assuming a weakly universal normal extension \(u\colon U\to B\) exists for \emph{every}~\(B\), we moreover have: 
\begin{proposition}\label{Proposition Fundamental group functor}
 The above definition of fundamental group gives a functor
 \[
 \pi_1(-,I)\colon {\C\to \Gp(\X)}.
 \]
\end{proposition}
\begin{proof}
Consider \(f\colon {A\to B}\) in \(\C\), and let~\(u\colon{U\to B}\) and \(v\colon {V\to A}\) be weak\-ly universal normal extensions of \(B\) and \(A\), respectively. Pulling back \(u\) along~\(f\) gives another normal extension of \(A\) by Lemma~\ref{Lemma Pullback Stable Normal}, so \(v\) factors over it, giving a morphism \(v\to u\) which need not be unique. However, Proposition~\ref{Proposition-Baer-Invariant} ensures that any two such morphisms induce the same morphism on~\(\pi_1\). It is then clear that \(\pi_1(-,I)\) preserves identities and composition.
\end{proof}

\begin{remark}
Not every Galois structure has the property that every object admits a weakly universal normal extension into it. Note, however, that even when this is not the case, the fundamental group still defines a functor, but its domain is restricted to the full subcategory of \(\C\) of those \(B\) for which \(\pi_1(B,I)\) is defined. \end{remark}

\begin{examples}
\begin{enumerate}
\item For the Galois structure \(\Gamma\) of Example \ref{Examples} \eqref{groups}, there is a weakly universal normal extension for every group \(B\): if \(p\colon {P\to B}\) is a surjective group homomorphism with a free domain \(P\), then the induced central extension \(P/[\Ker(p),P]\to B\) is easily seen to be weakly universal. The fundamental group \(\pi_1(B,I)=\H_2(B)\) in this case is the second (integral) homology group of \(B\). (See \cite{Janelidze:Hopf}.)

\item More generally, for Galois structures \(\Gamma\) of the type considered in Example \ref{Examples} \eqref{Mal'tsev}, \(\NExt(B)\) is a reflective subcategory of \((\E\downarrow B)\) for every \(B\) (see \cite{EverMaltsev, Janelidze-Kelly:Maltsev}), and the reflection into \(\NExt(B)\) of any regular epimorphism \(P\to B\) with a projective domain \(P\) is weakly universal. Hence, if \(\C\) is pointed with enough projectives, \(\pi_1(B,I)\) is well defined for every \(B\). 

When \(\C\) is a semi-abelian category with a monadic forgetful functor to \(\Set\), then \(\pi_1(B,I)=\H_2(B,I)\) is the second Barr-Beck homology group of \(B\) with coefficient functor \(I\) (see \cite{EGVdL}).

\item 
For Example \ref{Examples} \eqref{topological}, not every locally connected topological space~\(B\) admits a weakly universal normal extension \(u\colon {U\to B}\). However, it is well known that there exists a (surjective) covering map \(u\colon U\to B\) with a simply connected domain \(U\) for every connected, locally path-connected and semi-locally simply connected space~\(B\) (see, for instance, \cite{Hatcher,Spanier}). Such a \(u\) has the following property: for every covering map \(f\colon A\to B\) and every pair of elements \(x\in U\) and \(y\in A\) in corresponding fibres there is a unique continuous map \(a\colon U\to A\) (actually, a covering map) such that \(u=fa\) and \(a(x)=y\). Hence such a \(u\) is in particular a regular covering map which is clearly a weakly universal normal extension.

Choosing base points \(x\in U\) and \(y\in B\) such that \(u(x)=y\), the map \(u\colon (U,x)\to (B,y)\) becomes a weakly universal normal extension with respect to the Galois structure \(\Gamma\) of Example \ref{Examples} \eqref{pointedtopological}. In fact, in this case it is even an \emph{initial} object of \(\NExt(B)\) (rather than merely a weakly initial one), which agrees with the usual terminology of calling such a \(u\) a \emph{universal} covering map. Now \(\pi_1((B,y),I)\) is the classical Poincar\'e fundamental group of \((B,x)\) (see~\cite[Chapter 6]{Borceux-Janelidze}).
\end{enumerate}
\end{examples}

\section{The fundamental group functor as a Kan extension of the Galois group functor}\label{Section first Kan extension}
Throughout this section and the next, \(\Gamma=(\C,\X,H,I,\eta,\epsilon,\E,\F)\) will, as before, be an admissible Galois structure on a finitely complete pointed category \(\C\) with~\(H\) a full inclusion, and such that monadic extensions are pullback stable. For simplicity we shall moreover assume that every object of \(\C\) admits a weakly universal normal extension into it. However, our results can easily be adapted to situations where this is not the case (see Section \ref{Topological spaces}).

In the diagram
\[
\xymatrix{& \NExt(\C) \ar@{}[d]|-{\Nearrow}_{\kappa} \ar[ld]_-{\cod} \ar[rd]^-{\Gal(-,0)} \\
\C \ar@{.>}[rr]_-{\pi_{1}(-,I)} && \Gp(\X) }
\]
we now know all ingredients except the natural transformation \[\kappa\colon {\pi_1(-,I)\comp \cod\To \Gal(-,0)}.\] For a normal extension \(p\colon {E\to B}\), we define the component 
\[
\kappa_p\colon {\pi_1(B,I)\to \Gal(p,0)}
\]
as \(\Gal((h,1_B),0)\colon {\Gal(u,0)\to \Gal(p,0)}\) for a weakly universal normal extension \(u\colon {U\to B}\) and any induced
\[
 \xymatrix{U \ar@{..>}[r]^h \ar[d]_u & E \ar[d]^p\\ B \ar@{=}[r] &B}
\]
in \(\NExt(\C)\). Again by Proposition~\ref{Proposition-Baer-Invariant}, any such \((h,1_B)\) will induce the \emph{same} morphism \(\Gal((h,1_B),0)=\kappa_p\). It is easy to check that \(\kappa\) is natural.

To prove that the above diagram really is a Kan extension, we just have to show that this natural transformation \(\kappa\) is universal.

\begin{theorem}\label{First Theorem}
The following is a Kan extension:
 \[
\xymatrix{& \NExt(\C) \ar@{}[d]|-{\Nearrow}_{\kappa} \ar[ld]_-{\cod} \ar[rd]^-{\Gal(-,0)} \\
\C \ar@{.>}[rr]_-{\pi_{1}(-,I)} && \Gp(\X) }
\]
\end{theorem}
\begin{proof}
 Given another functor \(F\colon {\C\to \Gp(\X)}\) with a natural transformation
 \[
\xymatrix{& \NExt(\C) \ar@{}[d]|-{\Nearrow}_{\gamma} \ar[ld]_-{\cod} \ar[rd]^-{\Gal(-,0)} \\
\C \ar@{.>}[rr]_-{F} && \Gp(\X), }
\]
define \(\alpha\colon {F\To \pi_1(-,I)}\) by \(\alpha_B=\gamma_u\) for some weakly universal normal extension \(u\) of \(B\). This \(\alpha\) is really natural: given \(f\colon{A\to B}\) in \(\C\), the morphism
\[
\pi_1(f,I)\colon {\pi_1(A,I)\to \pi_1(B,I)}
\]
is defined as in Proposition~\ref{Proposition Fundamental group functor} using a morphism
\[
 \xymatrix{V \ar[d]_v \ar[r]^g & U \ar[d]^u\\ A \ar[r]_f & B}
\]
between weakly universal normal extensions of \(A\) and \(B\). Using naturality of \(\gamma\) on this morphism in \(\NExt(\C)\) gives naturality of \(\alpha\), because
\[\pi_1(f,I)=\Gal((g,f),0)\colon{\Gal(v,0)=\pi_1(A,I) \to \Gal(u,0)=\pi_1(B,I)}.
\]
Naturality of \(\gamma\) also implies that \(\kappa\comp \alpha_{\cod}=\gamma\): For each normal extension \(p\colon {E\to B}\), any morphism
\[
 \vcenter{\xymatrix{U \ar@{..>}[r]^h \ar[d]_u & E \ar[d]^p\\ B \ar@{=}[r] &B}}\qquad\text{ gives }\qquad
\vcenter{\xymatrix{ FB \ar@{=}[rr] \ar[d]_-{\gamma_u}|-{=}^-{\alpha_B} &&FB \ar[d]^-{\gamma_p} \\ \pi_1(B,I) \ar[rr]_-{\kappa_p=Gal((h,1_B),0)} && \Gal(p,0)}}
\]
and so \(\kappa_p\comp \alpha_B=\gamma_p\).

To see that \(\alpha\) is unique, notice that, for a weakly universal normal extension \(u\), the component \(\kappa_u\) is an isomorphism. So if \(\beta\colon {F\To \pi_1(-,I)}\) also satisfies \(\kappa\comp \beta_{\cod}=\gamma\), taking a weakly universal normal extension of \(B\) immediately implies \(\alpha_B=\beta_B\), for all \(B\).
\end{proof}

\begin{remark}
 In fact, in the definition of \(\pi_1(-,I)\) and the above proof of the universality of \(\kappa\), we have only used the following properties of \(\Gal(-,0)\) and \(\cod\):

 Given two functors
 \[
\xymatrix{& \N \ar[ld]_-{F} \ar[rd]^-{G} \\
\C \ && \Dc }
\] 
such that
\begin{enumerate}
 \item for all \(f,g\in \N\), \(F(f)=F(g)\) implies \(G(f)=G(g)\);
 \item for all \(C\in \C\), there exists \(U\in \N\) such that \(F(U)=C\) and, for all \(N\in \N\), the function 
 \[
 \Hom_\N(U,N)\to \Hom_\C(C,FN)
 \]
 giving the action of \(F\) is surjective.
\end{enumerate}
Then it is possible to define a functor \(H\colon {\C\to\Dc}\) via \(H(C)=G(U)\) and a natural transformation \(\kappa\colon {HF\To G}\) giving a Kan extension as we have done in our specific case above.
\end{remark}

\section{The fundamental group functor as a Kan extension of the kernel functor}\label{Section second Kan extension}

To compare this construction of the fundamental group given in the context of categorical Galois theory with other viewpoints on semi-abelian homology or with universality properties of connecting homomorphisms in long exact sequences, we actually need a slightly different Kan extension, namely
\[
\xymatrix{& \NExt(\C) \ar@{}[d]|-{\Nearrow}_(.45){\delta} \ar[ld]_-{\cod} \ar[rd]^-{\Ker} \\
\C \ar@{.>}[rr]_-{\pi_{1}(-,I)} && \X.}
\]
In this section we construct this Kan extension from the one we have already obtained. We first recall that the underlying object of a Galois group can also be calculated in another way:
\begin{lemma}\cite[Theorem 2.1]{Janelidze:Hopf}
Given a normal extension \(p\colon{E\to B}\), the underlying object of its Galois group can be computed as the intersection
\(
\Gal(p,0)=\Ker(p)\cap \Ker(\eta_E).
\)\noproof
\end{lemma}

This lemma implies that there is a component-wise monic natural transformation \[\iota\colon {U\comp \Gal(-,0)\To \Ker}\] from the functor giving the underlying object of the Galois group to the kernel functor. 
\[
 \xymatrix@!0@=6em{& \NExt(\C) \ar[dl]_{\cod} \ar[dr]|{U\comp\Gal(-,0)}^(.35)*[u]{\scriptstyle\iota}^(.4){\Nearrow} \ar@/^2pc/[dr]^{\Ker}& \\
 \C \ar@{..>}[rr]_{\pi_1(-,I)} \ar@{}[urr]_(.45){\Nearrow}|(.45){\kappa} && \X}
\]
It is clear that the big triangle in this diagram is still a Kan extension, forgetting only the internal group structure in the Kan extension of Section~\ref{Section first Kan extension}, since this internal group structure is not used anywhere in the proof. We now show that, for any functor \(F\colon {\C\to \X}\), any natural transformation \(\gamma\colon{F\comp \cod\To \Ker}\) factors over \(\iota\). Then universality of \(\kappa\) implies that \(\delta=\iota\comp \kappa\) also defines a Kan extension. However, we need a small extra condition to make this work: we now assume that 
\begin{quote}
\emph{all morphisms of the kind \(IE\to 0\) are in the class \(\F\).}
\end{quote}
Being split epimorphisms, this implies that they are monadic extensions (see \cite{Janelidze-Tholen2}), hence normal extensions, since the kernel pair projections are clearly trivial coverings, as they are in $\X$. Notice that this is indeed the case for all of our examples. 

\begin{lemma}
Let \(F\colon {\C\to \X}\) be a functor and \(\gamma\colon{F\comp \cod\To \Ker}\) a natural transformation.
 For any normal extension \(p\colon {E\to B}\), the component \(\gamma_p\) factors over the inclusion \(\Ker(p)\cap \Ker(\eta_E)\to \Ker(p)\).
\end{lemma}
\begin{proof}
 Since the above inclusion is the kernel of \(\xymatrix@1{\Ker(p) \ar[r]^-{\ker{p}} & E \ar[r]^-{\eta_E} & IE}\), it is sufficient to show that the composite
 \[
 \xymatrix@1{FB\ar[r]^-{\gamma_p} & \Ker(p) \ar[r]^-{\ker{p}} & E \ar[r]^-{\eta_E} & IE}
 \]
is zero. To do this, consider the three normal extensions
\[
 \xymatrix{E \ar[d]_p \ar[r]^{\eta_E} & IE \ar[d] & 0 \ar[d] \ar[l]\\
 B \ar[r] & 0 \ar@{=}[r] & 0 }
\]
with the given morphisms between them. Naturality of \(\gamma\) gives
\[
 \xymatrix{FB \ar[d]_{\gamma_p} \ar[r] & F0 \ar[d]^0 \ar@{=}[r] & F0 \ar[d] \\
 \Ker(p) \ar[r]_-{\eta_E\circ \ker{p}} & IE & 0 \ar[l],}
\]
which shows that \(\gamma_p\) does indeed factor over \(\Ker(p)\cap\Ker(\eta_E)\to \Ker(p)\).
\end{proof}

So, using universality of \(\kappa\) and this lemma, we obtain
\begin{theorem}\label{Theorem}
The diagram
 \[
\xymatrix{& \NExt(\C) \ar@{}[d]|-{\Nearrow}_(.45){\delta} \ar[ld]_-{\cod} \ar[rd]^-{\Ker} \\
\C \ar@{.>}[rr]_-{\pi_{1}(-,I)} && \X }
\]
is a Kan extension.\noproof
\end{theorem}

\section{When normal extensions are reflective}
Assume that \(\C\) is a semi-abelian category with enough regular projectives, that \(\X\) is a Birkhoff subcategory of \(\C\), and that \(\E\) and \(\F\) consist of all regular epimorphisms (so we are in the situation of Example \ref{Examples} \eqref{Mal'tsev}). It was shown in \cite{GVdL2} that there is a Kan extension 
 \[
\xymatrix{& \Ext_{\Gamma}(\C) \ar@{}[d]|-{\Nearrow}_(.45){\partial} \ar[ld]_-{\cod} \ar[rd]^-{\Ker\circ I_1} \\
\C \ar@{.>}[rr]_-{\pi_{1}(-,I)=\H_2(-,I)} && \X. }
\]
Here \(\Ext_{\Gamma}(\C)\) is the full subcategory of \(\Arr(\C)\) given by all monadic extensions,
\[
I_1\colon \Ext_{\Gamma}(\C)\to\NExt(\C)
\]
is left adjoint to the inclusion functor \(\NExt(\C)\to \Ext_{\Gamma}(\C)\) and, for every monadic extension \(p\colon E\to B\), the morphism \(\partial_{p}\colon \H_2(B,I)\to \Ker(I_1(f))\) is a connecting morphism in the long exact homology sequence associated with \(f\) and \(I\). In order to deduce this result from ours, we need a lemma.

\begin{lemma}\label{Kanreflection}
If the left hand triangle
 \[
\xymatrix{& \N \ar@{}[d]|-{\Nearrow}_(.45){\delta} \ar[ld]_-{F} \ar[rd]^-{G} \\
\C \ar@{.>}[rr]_-{K} && \Dc } \qquad \qquad
\xymatrix{& \M \ar@{}[d]|-{\Nearrow}_(.45){\delta_L} \ar[ld]_-{F\circ L} \ar[rd]^-{G\circ L} \\
\C \ar@{.>}[rr]_-{K} && \Dc }
\]
is a Kan extension and the functor \(L\colon \M\to \N\) admits a fully faithful right adjoint, then the right hand triangle is a Kan extension as well.
\end{lemma}
\begin{proof}
Write \(R\) for the fully faithful right adjoint of \(L\), and \(\epsilon\colon LR\To 1_{\M}\) for the counit. By \cite[Proposition 3 in X.7]{MacLane}, the natural transformation \(G\epsilon\colon GLR\To G\) defines a Kan extension, as pictured in the top triangle of the right hand diagram:
 
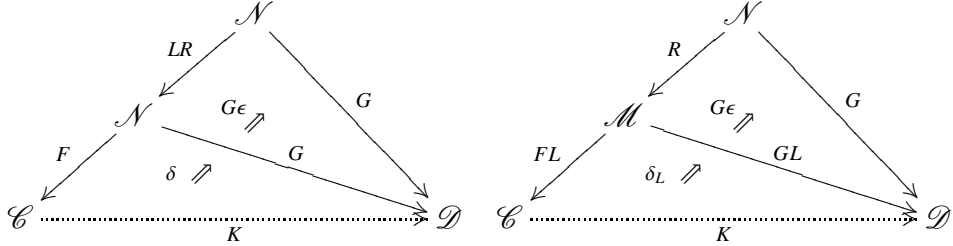
\[
\xymatrix{&& \N \ar@{}[dd]|-{\Nearrow}_(.45){G\epsilon} \ar[ld]_-{LR} \ar[rrdd]^-{G} \\
& \N \ar@{}[dr]|-{\Nearrow}_(.40){\delta} \ar[drrr]^G \ar[ld]_F &&\\
\C \ar@{.>}[rrrr]_-{K} &&&& \Dc } \;\;
\xymatrix{&& \N \ar@{}[dd]|-{\Nearrow}_(.45){G\epsilon} \ar[ld]_-{R} \ar[rrdd]^-{G} \\
& \M \ar@{}[dr]|-{\Nearrow}_(.40){\delta_L} \ar[drrr]^{GL} \ar[ld]_{FL} &&\\
\C \ar@{.>}[rrrr]_-{K} &&&& \Dc } 
\]
We want the bottom triangle in the right hand diagram to be a Kan extension as well. Since \(\epsilon\) is a natural isomorphism, this will be the case if the outer triangle and the natural transformation \({G\epsilon \comp \delta_{LR}}\colon {KFLR\To G}\) form a Kan extension. And indeed this is true, since the two outer triangles coincide, and in the left hand diagram both triangles are Kan extensions: the bottom one by assumption and the top one again by \cite[Proposition 3 in X.7]{MacLane}, because \(LR\colon \N\to \N\) is right adjoint to the identity functor, since \(R\) is fully faithful. 
\end{proof}

Theorem \ref{Theorem} and Lemma \ref{Kanreflection} imply in particular:
\begin{corollary}\label{oldKan}
Under the assumptions of Section \ref{Section second Kan extension}, and when, moreover, the inclusion functor \(\NExt(\C)\to \Ext_{\Gamma}(\C)\) admits a left adjoint
\[
I_1\colon \Ext_{\Gamma}(\C)\to \NExt(\C),
\]
the diagram
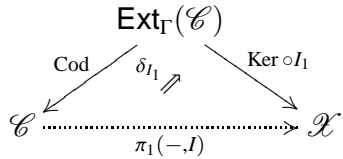
\[
\xymatrix{& \Ext_{\Gamma}(\C) \ar@{}[d]|-{\Nearrow}_(.45){\delta_{I_1}} \ar[ld]_-{\cod} \ar[rd]^-{\Ker\circ I_1} \\
\C \ar@{.>}[rr]_-{\pi_{1}(-,I)} && \X }
\]
is a Kan extension.
\end{corollary}
\begin{proof}
It suffices to observe that \(I_1\) leaves the codomains intact since every identity morphism is a normal extension and \(\NExt(\C)\) is a replete subcategory of \(\Ext_{\Gamma}(\C)\) (see Corollary 5.2 in \cite{Im-Kelly}).
\end{proof}
The inclusion functor \(\NExt(\C)\to \Ext_{\Gamma}(\C)\) admits a left adjoint not only in the semi-abelian case mentioned above, but more generally, when\-ever~\(\C\) is an exact Mal'tsev category, \(\X\) is a Birkhoff subcategory and \(\E\) and~\(\F\) consist of all regular epimorphisms (see \cite{EverMaltsev}). Another class of examples is given in \cite{Mathieu}.

\section{Exact homotopy sequence}\label{Topological spaces}
As remarked above, the Galois structure \(\Gamma\) of Example \ref{Examples} \eqref{pointedtopological} satisfies all conditions assumed in Sections \ref{Section first Kan extension} and \ref{Section second Kan extension}, except one: it is admissible, the category \(\LoCo_*\) is finitely complete and pointed, the discrete topology functor
\[
\dis\colon \Set_*\to \LoCo_*
\]
is fully faithful, monadic extensions are pullback stable and, for every pointed set \((X,x)\), the map \((X,x)\to 0\) is in \(\F\) (since here \(\F\) consists of \emph{all} base-point preserving maps); yet not every pointed topological space admits a weakly universal normal extension into it. We do know, however, that a \emph{universal} normal extension exists for every connected, locally path connected, semi-locally simply connected space \(B\) with base-point \(y\in B\), namely its universal covering map in the usual topological sense: a covering map \(u\colon (U,w)\to (B,y)\) with \(U\) connected and simply connected. Theorems \ref{First Theorem} and \ref{Theorem} and their proofs can easily be adapted to this situation. Thus we obtain Kan extensions
 \[
\xymatrix@C=1em{& \overline{\NExt(\LoCo_*)} \ar@{}[d]|-{\Nearrow}_{\kappa} \ar[ld]_-{\cod} \ar[rd]^-{\Gal(-,0)} \\
\overline{\LoCo}_* \ar@{.>}[rr]_-{\pi_{1}(-,\pi_0)} && \Gp}\quad 
\xymatrix@C=1em{& \overline{\NExt(\LoCo_*)} \ar@{}[d]|-{\Nearrow}_(.45){\delta} \ar[ld]_-{\cod} \ar[rd]^-{\Ker} \\
\overline{\LoCo}_* \ar@{.>}[rr]_-{\pi_{1}(-,\pi_0)} && \Set_* }
\]
where \(\overline{\LoCo}_*\) is the full subcategory of \(\LoCo_*\) consisting of all connected, locally path connected, semi-locally simply connected pointed spaces, and the full subcategory \(\overline{\NExt(\LoCo_*)}\) of \(\NExt(\LoCo_*)\) is determined by those  normal extensions whose codomain is in \(\overline{\LoCo}_*\). Notice that \(\Gp(\Set_*)\simeq\Gp(\Set)\simeq\Gp\). 

Now let \(p\colon (E,x)\to (B,y)\) be a \(\Gamma\)-normal extension of a connected, locally path-connected, semi-locally simply connected pointed space \((B,y)\) with kernel \((F,x)\) (meaning in this context of course the fibre over \(y\)). Let \(u\) be the universal covering map \((U,w)\to (B,y)\), write \(e\) for the unique continuous base-point preserving map \((U,w)\to (E,p)\) such that \(pe=u\) and recall that it is a covering map. Since \(U\) is connected, the image of \(e\) is contained in the connected component \(E_x\) of \(x\) and the left hand triangle restricts to the commutative right hand triangle 

\[
\xymatrix{
(U,w) \ar[d]_e \ar[rd]^u & \\
(E,x) \ar[r]_p & (B,y)}\quad\quad\quad\quad
\xymatrix{
(U,w) \ar[d]_{e'} \ar[rd]^u & \\
(E_x,x) \ar[r]_{p'} & (B,y)}
\]
Now \(e'\) is still a covering map, and it is surjective since its codomain is connected---the image of a covering map is always both open and closed. Moreover, since~\(U\) is connected and simply connected, \(e'\) is the universal covering map of \((E_x,x)\). Taking kernels yields an exact sequence of pointed sets
\[
0\to \Ker(e') \to \Ker(u) \to \Ker(p') \to 0
\] 
hence an exact sequence of groups
\[
0\to \pi_1(E,x)\to\pi_1(B,y)\to (F\cap E_x,x) \to 0
\]
where \((F\cap E_x,x)\) is the Galois group of the normal extension \(p'\). As we clearly have an exact sequence of pointed sets 
\[
0\to (F\cap E_x,x) \to (F,x) \to \pi_0(E,x) \to 0
\]
and because \((F,x)=\pi_0(F,x)\) since \(F\) is a discrete space, we can paste the two sequences together to obtain an exact sequence 
\[
0\to \pi_1(E,x)\to\pi_1(B,y)\to \pi_0(F,x) \to \pi_0(E,x) \to 0
\]
and this is the low-dimensional part of the usual exact homotopy sequence induced by the fibration 
\[
{(F,x)\to (E,x)\to (B,y)}.
\]
Notice that \(\pi_0(B,y)=0\) as \(B\) is connected. What we would like to point out here is that the morphism \(\pi_1(B,y)\to \pi_0(F,x)=(F,x)\) is the \(p\)-component \(\delta_p\) of the natural transformation defining the right hand Kan extension pictured above. Hence, we are in a similar situation as with the algebraic case studied in the previous section, where the Kan extension of Corollary \ref{oldKan} expresses a universal property of the connecting morphisms in an exact homology sequence.

%\bibliography{tim}
%\bibliographystyle{amsplain}

\providecommand{\noopsort}[1]{}
\providecommand{\bysame}{\leavevmode\hbox to3em{\hrulefill}\thinspace}
\providecommand{\MR}{\relax\ifhmode\unskip\space\fi MR }
% \MRhref is called by the amsart/book/proc definition of \MR.
\providecommand{\MRhref}[2]{%
  \href{http://www.ams.org/mathscinet-getitem?mr=#1}{#2}
}
\providecommand{\href}[2]{#2}

\small\noindent Tomas Everaert\\
Vakgroep Wiskunde, Vrije Universiteit Brussel\\
Pleinlaan 2, 1050 Brussel, Belgium\\
\textit{teveraer@vub.ac.be}

\vspace{2.5mm}

\noindent Julia Goedecke\\
Department of Pure Mathematics and Mathematical Statistics\\
University of Cambridge\\
Cambridge CB3 0WB, United Kingdom\\
\textit{julia.goedecke@cantab.net}

\vspace{2.5mm}

\noindent Tim Van~der Linden\\
Institut de Recherche en Math\'ematique et Physique\\
Universit\'e catholique de Louvain\\
chemin du cyclotron~2 bte L7.01.02, 1348 Louvain-la-Neuve, Belgium\\
\textit{tim.vanderlinden@uclouvain.be}

\end{document}